\documentclass[12pt,a4paper]{amsart}
\usepackage{amssymb,amsmath,amsthm}

\usepackage{bbm}

\usepackage{enumerate}
\usepackage{enumitem}

\setlist[itemize]{noitemsep, topsep=0pt}
\setlist[enumerate]{noitemsep, topsep=0pt}

\usepackage{tikz}
\usetikzlibrary{intersections}
\usetikzlibrary{positioning}

\newcommand{\mcal}{\mathcal}

\renewcommand{\phi}{\varphi}

\newcommand{\fin}{\textup{fin}}

\newcommand{\m}{\textup{m}}

\newcommand{\id}{\textup{id}}
\newcommand{\ON}{\textup{ON}}

\newcommand{\Part}{\textup{Part}}

\newcommand{\fix}{\textup{fix}}
\newcommand{\sym}{\textup{sym}}

\newtheorem{thm}{Theorem}[section]
\newtheorem{lem}[thm]{Lemma}
\newtheorem{cor}[thm]{Corollary}

\begin{document}

\title[The partitions and the permutations of a set]{The partitions whose members are finite and the permutations with at most $n$ non-fixed points of a set}
\author{Nattapon Sonpanow}
\author{Pimpen Vejjajiva}
\address{Department of Mathematics and Computer Science, Faculty of Science, Chulalongkorn University, Bangkok 10330, Thailand}

\email{nattapon.so@chula.ac.th}
\email{pimpen.v@chula.ac.th}


\date{}

\maketitle

\begin{abstract}
  We write  $S_{\leq n}(A)$ and $\Part_{\fin}(A)$  for
 the set of permutations with at most $n$ non-fixed points, where $n$ is a natural number, and the set of partitions  whose members are finite, respectively, of a set $A$.
  Among our results, we show, in the Zermelo-Fraenkel set theory, that  $|\Part_{\fin}(A)| \nleq |S_{\leq n}(A)|$ for any infinite set $A$ and if $A$ can be linearly ordered, then $|S_{\leq n}(A)| < |\Part_{\fin}(A)|$   while the statement ``$|S_{\leq n}(A)|\leq|\Part_{\fin}(A)|$ for all infinite sets $A$" is not provable  for  $n\geq 3$.
\end{abstract}

\section{Introduction}
 With the Axiom of Choice (AC), $|\Part(A)| = 2^{|A|}=|A|!$ for any infinite set $A$, where $\Part(A)$ is the set of partitions, $2^{|A|}$ and $|A|!$ are the cardinalities of the power set and the set of permutations of $A$ respectively. Without AC, it follows from the results in  \cite[Proposition 8.3]{HS2001} and \cite{DH} that these equalities are not provable in the Zermelo-Fraenkel set theory (ZF).

 Without AC, it is not hard to show that $2^{|A|} \leq |\Part(A)|$ for any set $A$ with $|A|\geq 5$.  Halbeisen and Shelah  showed in \cite[Theorem 3]{HS1994} that \lq\lq$|\fin(A)| < 2^{|A|}$ for any infinite set $A$\rq\rq\ is provable in ZF, where $\fin(A)$ is the set of finite subsets of $A$. As a result, $|\fin(A)| < |\Part(A)|$ for any infinite set $A$. A stronger result in \cite[Theorem 3.7]{PV} showed that  \lq\lq$|\fin(A)| < |\Part_{\fin}(A)|$ for any set $A$ with $|A|\geq 5$\rq\rq\ is provable in ZF, where  $\Part_{\fin}(A)$ is the set of partitions of $A$ whose members are finite. Moreover, \lq\lq$|\Part_{\fin}(A)|<2^{|A|}$ for some infinite set $A$\rq\rq\ is consistent with ZF (cf. \cite[Corollary 4.7]{PV}).

 For a set $A$, we write $S_{\fin}(A)$, $S_{\leq n}(A)$, and $S_{n}(A)$,  where $n$ is a natural number, for the sets of permutations on $A$ with finitely many non-fixed points, at most $n$ non-fixed points, and exactly $n$ non-fixed points respectively.
 The results in \cite[Theorems 2.9 and 2.10]{NPV} showed that, under AC$_{<\aleph_0}$, the Axiom of Choice for families of nonempty finite sets, $|S_{\fin}(A)|\leq|\fin(A)|$ if and only if $A$ is a Dedekind infinite set. Therefore, under AC$_{<\aleph_0}$, $|S_{\leq n}(A)|<|\Part_{\fin}(A)|$ for any Dedekind infinite set $A$.
 Shen and Yuan also showed in  ZF that $|S_{\leq n}(A)| < |A|!$ for any infinite set $A$ and any natural number $n$ such that $|A|>n>0$ (cf. \cite[Corollary 3.22]{SY}).

In this paper, we investigate relationship between $|\Part_{\fin}(A)|$  and $|S_{\leq n}(A)|$ for infinite sets $A$. We show, in ZF, that  $|\Part_{\fin}(A)| \nleq |S_{\leq n}(A)|$ for any infinite set $A$ and $|S_{\leq n}(A)| < |\Part_{\fin}(A)|$ if $A$ can be linearly ordered,  while the statement ``$|S_{n}(A)|\leq|\Part_{\fin}(A)|$ for all infinite sets $A$" is not provable in ZF for $n\geq 3$. In addition, while the result  in \cite[Theorem 3.2]{NV} showed that ``$|S_n(A)| \leq |S_{n+1}(A)|$ for all infinite sets $A$",  where $n>1$, is not provable in ZF, we show that it is provable if the subscript $n+1$ is replaced by any natural number $m$ such that $m\geq 2n$.

\section{Results in ZF}

In this section, we shall work in ZF without AC. We write $|A|$ for the \emph{cardinality} of a set $A$. For sets $A$ and $B$, we say $|A|=|B|$ if there is an explicit bijection from $A$ onto $B$, $|A|\leq|B|$ if there is an explicit injection from $A$ to $B$, and $|A|<|B|$ if $|A|\leq|B|$ but $|A|\neq|B|$. A set $A$ is \emph{Dedekind-infinite} if $\aleph_0\leq |A|$, otherwise $A$ is \emph{Dedekind-finite}.

Throughout, let $n$ be a natural number.

Apart from the notations introduced earlier, for a set $A$, let
\begin{enumerate}
\item $[A]^n=\{X\subseteq A :|X|=n\}$,
\item $\m(\phi)=\{ x\in A : \phi(x)\neq x \}$ where $\varphi$ is a permutation on $A$,
\item $(a_0; a_1;\ldots; a_n)$, where $a_0, a_1,\ldots, a_n$ are distinct elements of $A$, denote the cyclic permutation on A such that \[a_0\mapsto a_1\mapsto\ldots\mapsto a_n\mapsto a_0.\]
\end{enumerate}

First, we shall show that $|\Part_{\fin}(A)| \nleq |S_{\leq n}(A)|$ for any infinite set $A$.  The following facts are needed for the proof.

\begin{thm}\cite[Theorem 5.19]{{Hbs}}
For any infinite ordinal $\alpha$, $|\alpha|=|\fin(\alpha)|$.
\end{thm}

\begin{lem}\label{beta}
For any infinite ordinal $\alpha$, there is an infinite ordinal $\beta$ such that we can construct a bijection between $\alpha$ and $3\cdot\beta$.
\end{lem}

\begin{proof}
Let $\alpha$ be an infinite ordinal. Then $\alpha = 3\cdot \beta+\gamma$ for some infinite ordinal $\beta$ and some $\gamma<3$. It is clear if $\gamma=0$. Otherwise, define $f\colon 3\cdot\beta\to \alpha$ by
\begin{alignat*}{3}
f(n)&= 3\cdot\beta+n \quad&&\text{ for }n<\gamma,\\
f(n+\gamma)&=n &&\text{ for }n<\omega,\\
f(\xi)&=\xi &&\text{ for }\omega\leq \xi<3\cdot\beta.
\end{alignat*}
\noindent We can see that $f$ is bijective as desired.
\end{proof}

\begin{lem} \label{lem1}
For any natural number $k$, if $k\geq 2^{2n}$, where $n>0$, then $|\Part_{\fin}(k)| > |S_{\leq n}(k)|$.
\end{lem}

\begin{proof}
For any natural number $k>0$, since for each $X\in\mcal{P}(k)\setminus \{\emptyset,k\}$, $\{X,k\setminus X\}$ is a partition of $k$ and $\{k\}$ is also a partition of $k$, there are at least $2^{k-1}$ partitions of $k$. Thus $|\Part_{\fin}(k)| \geq 2^{k-1}> k^n$ for all $k \geq 2^{2n}$. Since each permutation in $S_{\leq n}(k)$ can be obtained by permuting $n$ elements chosen from $k$, for any $k\geq 2^{2n}$, we have
\[ |S_{\leq n}(k)| \leq {k \choose n} n! = k(k-1)\ldots(k-n+1) \leq k^n <|\Part_{\fin}(k)|.\]
\end{proof}

\begin{thm} \label{BigThm}
For any infinite set $A$, $|\Part_{\fin}(A)| \nleq |S_{\leq n}(A)|$.
\end{thm}

\begin{proof}
Let $A$ be an infinite set. It is trivial if $n\leq 1$. Suppose there is an injection $F \colon  \Part_{\fin}(A) \rightarrow S_{\leq n}(A)$ where $n\geq 2$. We shall show that for any infinite ordinal $\alpha$, we can construct a one-to-one sequence of members of $\Part_{\fin}(A)$ with length $\alpha$, which contradicts Hartogs' Theorem.

In order to construct such one-to-one sequence with length $\omega$, we shall construct a family $\{ A_i : i\in\omega \} \subseteq \fin(A)$ such that $A_{i} \subsetneq A_{i+1}$ for all $i\in\omega$.

Pick $A_0 \subseteq A$ such that $|A_0|=2^{2n}$. Suppose we have already defined a finite set $A_i$ such that $A_0 \subseteq A_i \subseteq A$. By Lemma \ref{lem1},
$$ \left| \Part^{(A)}_{\fin}(A_i) \right| = \left| \Part_{\fin}(A_i) \right| > \left| S_{\leq n}(A_i) \right| = \left| S^{(A)}_{\leq n}(A_i) \right|, $$
where
\begin{align*}
\Part^{(A)}_{\fin}(A_i) &= \{ \Pi \cup [A\setminus A_i]^1 : \Pi \in \Part_{\fin}(A_i) \} \text{\ and} \\
S^{(A)}_{\leq n}(A_i) &= \{ \phi \cup \id_{A\setminus A_i} : \phi \in S_{\leq n}(A_i) \}.
\end{align*}
Let $B_i = F\left[ \Part^{(A)}_{\fin}(A_i) \right]\setminus S^{(A)}_{\leq n}(A_i) $. Since $F$ is injective, $\emptyset \neq B_i \subseteq S_{\leq n}(A)$. Moreover, $\m(\phi) \not\subseteq A_i$ for all $\phi \in B_i$. Note that $B_i$ is finite since $\Part^{(A)}_{\fin}(A_i)$ is finite. We define
\[A_{i+1} = A_i \cup \bigcup\{ \m(\phi) : \phi \in B_i \}.\]
Then $A_{i+1} \supsetneq A_i$ and $A_{i+1}$ is still a finite subset of $A$.

Now, for each $n\in\omega$, let \[P_n= \{ A_n \} \cup [A\setminus A_n]^1.\] We can see that $\langle P_0, P_1, \ldots, P_n, \ldots\rangle_\omega$ is a one-to-one sequence as desired.

Next, assume there is a one-to-one sequence $\langle \Pi_0, \Pi_1, \ldots, \Pi_i, \ldots\rangle_\alpha$ of members of $\Part_{\fin}(A)$  where the length  $\alpha$ is an infinite ordinal.

For each $i<\alpha$, let $\phi_i = F(\Pi_i)$ and $A_\alpha = \bigcup \{\m(\phi_i) : i<\alpha\}$.

 Define an equivalence relation $\sim$ on $A_\alpha$ by
\[x \sim y\text{ if and only if }\forall i<\alpha(x\in \m(\phi_i) \leftrightarrow y\in \m(\phi_i)).\]
Obviously,  $\left| [x]_\sim \right| \leq n$ for all $x\in A_\alpha$. Next, we shall show that there is a bijection between $\{ [x]_\sim : x\in A_\alpha \}$ and $\alpha$ by using the idea from the proof of \cite[Theorem 3]{HS1994}.

For each $x\in A_\alpha$ and $\mu\leq\alpha$, define
\[D_{x,\mu} = \bigcap\{ \m(\phi_i) : i<\mu\text{ and }x\in \m(\phi_i) \},\]
where $D_{x,\mu} = A_\alpha$ if $x\notin \m(\phi_i)$ for all $i<\mu$ and define
\[g_x = \{ \iota<\alpha : x\in \m(\phi_\iota)\text{ and } D_{x,\iota+1} \subsetneq D_{x,\iota} \}.\]
It is easy to see that, for any $x,y\in A_\alpha$,  $x \sim y$ implies $g_x = g_y$. For the converse, we can see that for any $x,y\in A_\alpha$, if $\nu<\alpha$ is the least ordinal such that $x \in \m(\phi_\nu)$ but $y \notin \m(\phi_\nu)$, then $D_{x,\nu+1} \subsetneq D_{x,\nu} = D_{y,\nu} = D_{y,\nu+1}$, which implies $\nu \in g_x\setminus g_y$, so $g_x\neq g_y$. Note that for each $x\in A_\alpha$, since $D_{x,\mu}$ is finite for any $\mu\leq\alpha$, $g_x \in \fin(\alpha)$.  Thus, by sending each $[x]_\sim$ to $g_x$, we get an injection from $K  = \{ [x]_\sim : x\in A_\alpha \}$ to $\fin(\alpha)$. As $|\fin(\alpha)|=|\alpha|$, we can construct an injection from $K $ to $\alpha$. So $K $ has a well order induced by $\alpha$ with an order type, say $\gamma$. Hence, there is a bijection $p\colon  K  \rightarrow \gamma$. Since $K $  is infinite, so is $\gamma$.
Note that, since $F$ is injective, for each $i<\alpha$, $\{j<\alpha : \m(\phi_i)=\m(\phi_j)\}$ has at most $n!$ elements.
Thus, the map $q \colon \alpha \rightarrow \fin(\gamma)\times\gamma$  defined by
\[q(i) = (\{ p([x]_\sim) : x\in \m(\phi_i)\}, k_i),\]
where $i$ is the $k_i$th ordinal in the set $\{j<\alpha : \m(\phi_i)=\m(\phi_j)\}$, is an injection.
Since $|\fin(\gamma)\times \gamma|=|\gamma|$, we obtain an injection from $\alpha$ to $\gamma$.
Since $\gamma\leq\alpha$, by the Cantor-Bernstein Theorem, we get a bijection between $\alpha$ and $\gamma$, and thus we can construct a bijection $h \colon  \alpha \rightarrow K $.

By Lemma \ref{beta}, there are an infinite ordinal $\beta$ and a bijection $t \colon  3\cdot\beta \rightarrow \alpha$. So  $H = h \circ t \colon  3\cdot\beta \rightarrow K$ is a bijection.

Next, we shall construct $\Pi_{\alpha} \in \Part_{\fin}(A)$ which is distinct from $\Pi_i$ for all $i<\alpha$.

Let us fix $\delta<\beta$ and consider the following four partitions of $\bigcup \{ H(3\cdot\delta+j) : j<3 \}$:
\begin{align*}
C_{0}^{\delta} &= \{ H(3\cdot\delta), H(3\cdot\delta+1)\cup H(3\cdot\delta+2) \}, \\
C_{1}^{\delta} &= \{ H(3\cdot\delta+1), H(3\cdot\delta+2)\cup H(3\cdot\delta) \}, \\
C_{2}^{\delta} &= \{ H(3\cdot\delta+2), H(3\cdot\delta)\cup H(3\cdot\delta+1) \}, \\
C_{3}^{\delta} &= \{ \textstyle{\bigcup\{H(3\cdot\delta+j) : j<3\}} \}.
\end{align*}
Pick the least $m<4$ such that $C_{m}^{\delta} \not\subseteq \Pi_{t(3\cdot\delta+j)}$ for all $j<3$ and write $C_\delta$ for this $C_{m}^{\delta}$. After $C_\delta$'s are obtained for all $\delta<\beta$, we define\[\Pi_{\alpha} = \bigcup\{ C_\delta : \delta<\beta \} \cup [A\setminus A_\alpha]^1.\]
Note that $\bigcup\{ C_\delta : \delta<\beta \}\in \Part_{\fin}(A_\alpha)$ and so $\Pi_\alpha\in\Part_{\fin}(A)$.

For each $\delta<\beta$ and $j<3$, we have that
$C_\delta \subseteq \Pi_{\alpha}$ but $C_\delta \not\subseteq \Pi_{t(3\cdot\delta+j)}$,
so $\Pi_{\alpha} \neq \Pi_{t(3\cdot\delta+j)}$, which means $\Pi_{\alpha}$ is distinct from  $\Pi_i$ for all $i<\alpha$.
 Then we obtain a one-to-one sequence of members of $\Part_{\fin}(A)$ with length $\alpha+1$.

We can see that the sequence constructed by the above process is an extension of the sequence previously constructed. Thus, we can define a sequence whose length is a limit ordinal as the union of all sequences constructed earlier.
\end{proof}

\begin{cor} \label{S2}
For any infinite set $A$, $|S_{\leq2}(A)| < |\Part_{\fin}(A)|$.
\end{cor}

\begin{proof}
This follows from Theorem \ref{BigThm} since for any nonempty set $A$, the map $F\colon S_{\leq2}(A)\to\Part_{\fin}(A)$ defined  by \[F(\varphi)=\begin{cases}\{ \m(\phi) \} \cup [A\setminus\m(\phi)]^1&\text{if }\m(\phi)\neq\emptyset \\
[A]^1&\text{otherwise,}\end{cases}\] is injective.
\end{proof}

It follows from the results in \cite[Theorem 2.9]{NPV} and \cite[Theorem 3.7]{PV} that if AC$_{<\aleph_0}$ is assumed, then $|S_{\leq n}(A)| < |\Part_{\fin}(A)|$ for any Dedekind-infinite set $A$. Note that AC$_{<\aleph_0}$ is weaker than the Ordering Principle which states that \lq\lq every set can be linearly ordered\rq\rq\ (cf. \cite{Howard} and \cite[page 104]{Jech}) but the statement \lq\lq every infinite set is Dedekind-infinite\rq\rq\ is independent from the Ordering Principle (cf. \cite{Howard}). However, we obtain the same result for infinite linearly ordered  sets.

\begin{thm}\label{linear}
For any infinite linearly ordered set $A$, $|S_{\leq n}(A)| < |\Part_{\fin}(A)|$.
\end{thm}

\begin{proof}
Let $A$ be an infinite set with  a linear order $\lhd$.  By Theorem \ref{BigThm}, it suffices to show that $|S_{\leq n}(A)| \leq |\Part_{\fin}(A)|$.
The result is trivial for $n<2$. Assume $n\geq 2$.  Define $g\colon S_{\leq n}(A) \rightarrow S_{\leq n}(n)$ as follows:

For $\phi \in S_{\leq n}(A)$ with $\m(\phi) = \{a_0,a_1,\ldots ,a_{\ell-1}\}$ where $a_0 \lhd a_1 \lhd \ldots  \lhd a_{\ell-1}$, define \[g(\phi) = \{ (x,y) \in \ell\times\ell : (a_x,a_y)\in \phi \} \cup \id_{n \setminus \ell}.\]

Let $p = |S_{\leq n}(n)|$, $f\colon S_{\leq n}(n) \rightarrow p$ be a bijection, $D=\{ X\subseteq A : 1\neq|X|\leq n\}$, and define $h\colon S_{\leq n}(A) \rightarrow D \times p$ by\[h(\phi) = (\m(\phi), f(g(\phi))).\]
To see that $h$ is injective, suppose $\phi_1,\phi_2 \in S_{\leq n}(A)$ are such that $\phi_1 \neq \phi_2$ and $\m(\phi_1) = \m(\phi_2)=\{a_0,a_1,\ldots ,a_{\ell-1}\}$ where $a_0 \lhd a_1 \lhd \ldots  \lhd a_{\ell-1}$. Since $\phi_1\neq\phi_2$, there is some $(x,y) \in \ell\times\ell$ such that $(a_x,a_y) \in (\phi_1\setminus\phi_2)\cup (\phi_2\setminus\phi_1)$, which implies  $g(\phi_1) \neq g(\phi_2)$, and hence $f(g(\phi_1)) \neq f(g(\phi_2))$ since $f$ is injective.

Next, we shall construct an injection $F\colon D \times p \rightarrow \Part_{\fin}(A)$. First, fix $p(n+1)^2$ members of $A$ and divide them into pairwise disjoint family $\{B^i_j\subseteq A : i\leq n\text{ and }j<p\}$ such that $|B^i_j|=n+1$ for each $i,j$. For each $(X,j) \in D \times p$, define
\[F(X,j) =\begin{cases} \{X,B^{k}_j\} \cup [A\setminus(X\cup B^{k}_j)]^1&\text{if } X\neq\emptyset,\\
\{B^{k}_j\} \cup [A\setminus  B^{k}_j]^1&\text{otherwise},\end{cases}\]
where $k=\min\{ i\leq n : B^i_j \cap X = \emptyset \}$.
We can see that $F$ is injective, so $F \circ h \colon  S_{\leq n}(A) \rightarrow \Part_{\fin}(A)$ is an injection.
\end{proof}

From \cite[Theorem 3.2]{NV}, we know that \lq\lq$|S_n(A)| \leq |S_{n+1}(A)|$ for any infinite set $A$\rq\rq\  is not provable in ZF for $n>1$. Surprisingly, the statement is provable when $n+1$ is replaced by some large enough natural numbers.

\begin{thm}\label{Sn}
For any infinite set $A$ and any natural number $m$, if $m\geq 2n$, where $n>1$, then $|S_n(A)| \leq |S_m(A)|$.
\end{thm}

\begin{proof}
Let $A$ be an infinite set and $m$ be a natural number such that $m\geq 2n$, where $n>1$. First, we assume $m>2n$. Fix a sequence of $m$ distinct members of $A$, say $\langle y_0,y_1,\ldots,y_{m-1}\rangle$.  For each $\phi \in S_n(A)$, define a permutation $F_\phi \in S_{m-n}(A)$ by \[F_\phi = (x_0;x_1;\ldots ;x_{m-n-1}),\] where $x_0,x_1, \ldots, x_{m-n-1}$ are the first $m-n$ entries of  $\langle y_0,y_1,\ldots,y_{m-1}\rangle$ which are not in $\m(\phi)$ and define $G\colon S_n(A) \rightarrow S_m(A)$ by
\[G(\phi) = \phi \circ F_\phi.\]
Observe that for each $\phi \in S_n(A)$, $\m(\phi) \cap \m(F_\phi) = \emptyset$ and, as $|\m(F_\phi)|=m-n>n=|\m(\phi)|$, $F_\phi$ is the only largest cycle in $G(\phi)$. To see that $G$ is injective, let $\varphi_1$, $\varphi_2\in S_n(A)$ be such that $G(\phi_1) = G(\phi_2)$. Then their largest cycles are the same, which means $F_{\phi_1} = F_{\phi_2}$. So $\phi_1 = G(\phi_1) \circ F_{\phi_1}^{-1} = G(\phi_2) \circ F_{\phi_2}^{-1} = \phi_2$.

Now, assume $m=2n$.
Fix an $n(n+1)$-element subset of $A$, say $B = \{x^i_j : i\leq n\text{ and }j<n\}$. Let $B_i = \{x^i_0, x^i_1, \ldots , x^i_{n-1}\}$ and $\chi_i=(x^i_0;x^i_1;\ldots ;x^i_{n-1})\in S_n(A)$ for each $i\leq n$.

For each $\phi \in S_n(A)$, if there is some (unique) $i\leq n$ such that $\m(\phi)= B_i$, then let $N_\phi = i+1$ (modulo $n+1$), otherwise let $N_\phi = \min\{i\leq n : \m(\phi) \cap B_i = \emptyset\}$. Define $G\colon S_n(A) \rightarrow S_{2n}(A)$ by \[G(\phi) = \phi \circ \chi_{N_\phi}.\]
Observe that for each $\phi\in S_n(A)$, $\m(\phi) \cap \m(\chi_{N_\phi}) = \emptyset$ and $|\m(\phi)|=|\m(\chi_{N_\phi})|=n$, so $G(\phi) \in S_{2n}(A)$. To show that $G$ is injective, suppose $\phi,\psi \in S_n(A)$ are such that  $G(\phi) = G(\psi)$.

Suppose $\phi$ and $\psi$ are cycles with $\m(\phi)=B_p$ and $\m(\psi)=B_q$ for some $p,q\leq n$. Then each of $G(\phi)$ and $G(\psi)$ is a product of two disjoint cycles with length $n$. Suppose $\phi \neq \psi$. Then $\phi = \chi_{N_\psi}$ and $\psi = \chi_{N_\phi}$ by the uniqueness of the decompositions of $G(\phi)$ and $G(\psi)$. From $\phi = \chi_{N_\psi}$, we have $\m(\phi) = B_{N_\psi}$. By the definition of $N_\phi$, we have that $N_\phi = N_\psi+1$ (modulo $n+1$). Similarly, $N_\psi = N_\phi+1$ (modulo $n+1$). So $0=2$ (modulo $n+1$), which is impossible since $n\geq 2$.

For the remaining cases, we may assume $\psi$ is not a cycle or $\m(\psi) \neq B_i$ for all $i\leq n$. Since  $\chi_{N_\phi}$ is a cycle and $\m(\chi_{N_\phi})=B_{N_\phi}$,  $\chi_{N_\phi}\neq \psi$.  Thus $\chi_{N_\phi} = \chi_{N_\psi}$ which implies $\phi=\psi$.
\end{proof}

\section{Consistency Results and Summary}

We have shown, in ZF, that if $X$ is an infinite linearly ordered set, then $|S_{\leq n}(X)| < |\Part_{\fin}(X)|$ (Theorem \ref{linear}).  For arbitrary infinite sets $X$, we show that $|S_{\leq2}(X)| < |\Part_{\fin}(X)|$ (Corollary \ref{S2}) and  $|\Part_{\fin}(X)| \nleq |S_{\leq n}(X)|$ (Theorem \ref{BigThm}). Now, we shall show that the latter statement is the best possible result in ZF for arbitrary infinite sets $X$ and for $n\geq 3$.

We shall use permutation models, which are models of ZFA, set theory with atoms. This theory admits objects which are not sets, called \emph{urelements} or \emph{atoms}. We provide sufficient details as follows:

Let $A$ be an infinite set of atoms and $\mcal{G}$ be a group of permutations on $A$. Define $V_0=A$, $V_{\alpha+1}=\mcal{P}(V_\alpha) \cup V_\alpha$, $V_{\gamma} = \bigcup_{\alpha<\gamma} V_\alpha$ for limits $\gamma$, and $V = \bigcup_{\alpha \in \ON} V_\alpha$. Each $\pi\in\mcal{G}$ is extended to a permutation on $V$ so that $\pi x=x$ whenever $x$ is a \emph{pure set}, a set whose transitive closure contains no atoms.  For each $x\in V$, let $\fix_\mcal{G}(x) = \{ \pi\in\mcal{G} : \pi y = y \text{\ for all\ } y\in x \}$ and $\sym_\mcal{G}(x) = \{ \pi\in\mcal{G} : \pi x = x \}$. For a normal ideal $I$ on $A$, a set $E \in I$ is a \emph{support} of $x$ if $\fix_\mcal{G}(E) \subseteq \sym_\mcal{G}(x)$. Given a normal ideal $I$ on $A$, we define $\mcal{V} = \{x\in V : x \text{\ has a support and\ } x\subseteq\mcal{V}\}$. The class $\mcal{V}$, which is determined by $A$, $\mcal{G}$, and $I$, is called a \emph{permutation model}. For more details, see \cite[Chapter 4]{Jech}.
We shall use the basic Fraenkel model $\mcal{V}_{F_0}$ which is a permutation model with a countably infinite set $A$ of atoms, the group $\mcal{G}$ of all permutations on $A$, and the normal ideal $\fin(A)$.

 In the following, we assume $n\geq 3$ and for a set $X$, let $C_{n}(X)=\left\{ \phi \in S_n(X) : \phi \text{\ is a cycle}\right\}$.

\begin{thm}
$\mcal{V}_{F_0} \vDash  |C_{n}(A)| \nleq |\Part(A)|$.
\end{thm}

\begin{proof}
 Suppose to the contrary that there is an injection $F: C_n(A) \rightarrow \Part(A)$ with a finite support $E$. Pick $n$ distinct elements $a_0,a_1,\ldots ,a_{n-1}$ in $A\setminus E$. Define $\pi=(a_0;a_1;\ldots; a_{n-1})$ and let $\Pi = F(\pi)$. Since $\pi$ fixes all members of $E$ and $\pi\pi = \pi$,
\[\pi\Pi=\pi(F\pi)=\pi F(\pi\pi)=F(\pi)=\Pi.\]

For $a,b\in A$, we say that $a \sim_\Pi b$ whenever there is some $X\in \Pi$ such that $a,b\in X$, and write $[a]_\Pi$ for $\{c : a \sim_\Pi c\}$.

\textbf{Case 1.} There are some distinct $i,j<n$ such that $a_i \sim_\Pi a_j$.

Since $\pi \Pi = \Pi$, we have $\pi a_i \sim_\Pi \pi a_j$.
 So there is some $P \in \Pi$ such that $\pi a_i,\pi a_j \in P$. Define $\rho=(\pi a_i;\pi a_j)$. Since $\m(\rho)=\{\pi a_i,\pi a_j\}\subseteq P\setminus E$, $\rho F = F$ and $\rho \Pi = \Pi$. Note that $\rho \pi\in C_n(A)$.  Hence
 \[F(\rho \pi) = \rho F(\rho \pi)= \rho(F\pi)= \rho \Pi = \Pi = F(\pi).\] However, $\rho \pi \neq \pi$ since  $n\geq 3$. This contradicts the injectivity of $F$.

 \medskip

\textbf{Case 2.} There are no distinct $i,j<n$ such that $a_i \sim_\Pi a_j$.

Suppose there is some $i<n$ such that $[a_i]_\Pi\neq\{a_i\}$, i.e. $\{a_i\}\subsetneq [a_i]_\Pi$. Let $X=[a_i]_\Pi\setminus\{a_i\}$. Then $\emptyset\neq X\subseteq A\setminus\{a_k:k<n\}$ and so $\pi$ fixes $X$ pointwise. Since $\pi\Pi=\Pi$, $X\subseteq[a_i]_\Pi\cap [\pi a_i]_\Pi$ but $a_i \not\sim_\Pi \pi a_i$, a contradiction. Thus $[a_j]_\Pi=\{a_j\}$ for all $j<n$.
Let $\rho = (a_0;a_1)$. Then $\rho \Pi = \Pi$ and hence, as in the previous case, $F(\rho \pi) = F(\pi)$ while $\rho \pi \neq \pi$, contradicting the injectivity of $F$.
\end{proof}

From the above theorem, we have that \lq\lq$|C_{n}(X)| \nleq |\Part(X)|$ for some infinite set $X$\rq\rq\ holds in $\mcal{V}_{F_0}$ which is a model of ZFA. This statement can be transferred to ZF by the Jech-Sochor First Embedding Theorem (cf. \cite[Theorem 6.1]{Jech}). As a result, \lq\lq$|C_{n}(X)| \leq |\Part(X)|$ for all infinite sets $X$\rq\rq\ is not provable in ZF, provided ZF is consistent.
Since $C_n(X) \subseteq S_{n}(X)\subseteq S_{\leq n}(X)$ and $\Part_{\fin}(X)\subseteq \Part(X)$ for any set $X$, $C_n(A)$ in the above theorem can be replaced by $S_{n}(A)$ and $S_{\leq n}(A)$, and $\Part(A)$ by $\Part_{\fin}(A)$ as well. Therefore we can conclude that the best possible result for relationships between $S_{\leq n}(X)$ and  $\Part_{\fin}(X)$ provable in ZF for arbitrary infinite sets $X$ and for $n\geq3$ is that $|\Part_{\fin}(X)| \nleq |S_{\leq n}(X)|$.

For relationships among $S_n(X)$'s for infinite sets $X$, where $n>1$, it has been shown in \cite[Theorem 3.2]{NV} that \lq\lq$S_n(X)\leq S_{n+1}(X)$ for all infinite sets $X$\rq\rq\ is not provable in ZF. However, we show in Theorem \ref{Sn} that the statement is provable if the subscript $n+1$ is replaced by any natural number $m$ such that $m\geq2n$. We are still wondering whether the subscript $n+1$ can be replaced by $2n-1$ or not. This is left open for future research.

\section{Acknowledgement}

This research project is supported by grants for development of new faculty staff, Ratchadaphiseksomphot Fund, Chulalongkorn University.

\end{document}